\documentclass[preprint,12pt]{elsarticle}

\usepackage{amsfonts,amsthm,amsmath}
\usepackage{mathrsfs}

\usepackage[margin=3cm]{geometry}
\usepackage{amsthm,amsmath,amssymb}
\usepackage{fullpage}

\date {}

\newtheorem{thm}{Theorem}

\newtheorem{lemma}[thm]{Lemma}
\newtheorem{corollary}[thm]{Corollary}
\newtheorem{remark}[thm]{Remark}

\theoremstyle{definition}

\newtheorem{question}{Question}

\newcommand{\nats}{{\mathbb N}}

\newcommand{\reals}{{\mathbb R}}

\newcommand{\U}{\mathcal{U}}
\newcommand{\card}{{\rm Card}}
\def\A{\mathbb{A}}

\newcommand{\Card}{{\rm Card}}

\newcommand{\F}{\mathscr{F}}

\begin{document}

\begin{frontmatter}

\title{Monochromatic factorisations of words and periodicity}

\author[label1]{Ca\"ius Wojcik}
\ead{caius.wojcik@gmail.com}

   \author[label1]{Luca Q. Zamboni}
  \ead{zamboni@math.univ-lyon1.fr}

\address[label1]{Universit\'e de Lyon,
Universit\'e Lyon 1, CNRS UMR 5208,
Institut Camille Jordan,
43 boulevard du 11 novembre 1918,
F69622 Villeurbanne Cedex, France}



\begin{abstract}
In 2006 T. Brown asked the following question in the spirit of Ramsey theory:  Given a non-periodic infinite word $x=x_1x_2x_3\cdots$ with values in a non-empty set $\A,$ does there exist a finite colouring $\varphi: \A^+\rightarrow C$ relative to which $x$ does not admit a $\varphi$-monochromatic factorisation, i.e., a factorisation of the form $x=u_1u_2u_3\cdots$ with $\varphi(u_i)=\varphi(u_j)$ for all $i,j\geq 1$?  This question belongs to the class of Ramsey type problems in which one asks whether some abstract form of Ramsey's theorem holds in a certain generalised setting.
Various partial results in support of an affirmative answer to this question have appeared in the literature in recent years. In particular it is known that the question admits an affirmative answer for all non-uniformly recurrent words and hence for  almost all words relative to the standard Bernoulli measure on $A^\nats.$ This question also has a positive answer for various classes of uniformly recurrent words including Sturmian words and fixed points of strongly recognizable primitive substitutions. In this paper we give a complete and optimal affirmative answer to this question by showing that if $x=x_1x_2x_3\cdots$ is an infinite non-periodic word with values in a non-empty set $\A,$ then there exists a $2$-colouring $\varphi: \A^+\rightarrow \{0,1\}$ 
such that for any factorisation $x=u_1u_2u_3\cdots$ we have $\varphi(u_i)\neq \varphi(u_j)$ for some $i\neq j.$ In fact this condition gives a characterization of non-periodic words.  It may be reformulated in the language of ultrafilters as follows:  Let $\beta \A^+$ denote the Stone-\v Cech compactification of the discrete semigroup $\A^+$ which we regard as  the set of all ultrafilters on $\A^+.$ Then an infinite word $x=x_1x_2x_3\cdots$ with values in  $\A$ is periodic if and only if there exists $p\in \beta \A^+$ such that for each $A\in p$ there exists a factorisation $x=u_1u_2u_3\cdots $ with each $u_i \in A.$ Moreover $p$ may be taken to be an idempotent element of $\beta \A^+.$

 \end{abstract}

\journal{}

\end{frontmatter}

\section{Introduction}

Give a  non-empty (not necessarily finite) set $\A,$  let $\A^+$ denote the free semigroup generated by $\A$ consisting of all finite words $u_1u_2\cdots u_n$ with $u_i \in \A,$ and let $\A^\nats$ denote the set of all infinite words $x=x_1x_2x_3\cdots$ with $x_i\in \A.$ 
We say $x\in \A^\nats$ is {\it periodic} if $x=u^\omega=uuu\cdots $ for some $u\in \A^+.$ 
The following question was independently posed by T. Brown in \cite{BTC} and by the second author in \cite{LQZ}\footnote{The original formulation of the question was stated in terms of finite colourings of the set of all factors of $x.$}:

\begin{question}\label{conj} Let $x\in \A^\nats$ be non-periodic. Does there exist a finite colouring $\varphi: \A^+\rightarrow C$ relative to which $x$ does not admit a $\varphi$-monochromatic factorisation, i.e., a factorisation of the form $x=u_1u_2u_3\cdots$ with $\varphi(u_i)=\varphi(u_j)$ for all $i,j\geq 1$?
\end{question}

A finite colouring $\varphi: \A^+\rightarrow C$ is called a {\it separating colouring} for $x$ (or a separating $|C|$-colouring for $x)$ if for all factorisations $x=u_1u_2u_3\cdots$ there exist $i,j\geq 1$ such that  $\varphi(u_i)\neq \varphi(u_j).$  While T. Brown originally stated it as a question, Question~\ref{conj} has evolved into a conjecture which states that every non-periodic word admits a separating colouring. 
We begin by illustrating Question~\ref{conj} with an example: Consider the {\it Thue-Morse} infinite word \[x=011010011001011010010\cdots\] where the $n$th term of $x$ (starting from $n=0)$  is defined as the sum modulo $2$ of the digits  in the binary expansion of
$n.$  The origins of this word go back to the beginning of the last century with the works of  A. Thue \cite{Th1, Th2} in which he proves amongst other things that $x$ is {\it overlap-free} i.e., $x$ contains no factor of the form $uuu'$ where $u'$ is a non-empty prefix of $u.$ 

Consider the $3$ colouring  $\varphi : \{0,1\}^+\rightarrow \{0,1,2\}$ defined by 
$$ \varphi(u) = \begin{cases}
0 & \text{if $u$ is a prefix of $x$ ending with $0;$} \\ 1 &  \text{if $u$ is a prefix of $x$ ending with $1;$}\\
2& \text{if $u$ is not a prefix of $x.$}
\end{cases} $$
We claim that no factorisation of $x$ is $\varphi $-monochromatic. In fact, suppose to the contrary that $x=u_1u_2u_3\cdots$ is a  $\varphi $-monochromatic factorisation of $x.$ Since $u_1$ is a prefix of $x$, it follows that $\varphi(u_1)\in \{0,1\},$ i.e., there exists $a\in \{0,1\}$ such that   each $u_i$ is a  prefix of $x$ terminating with $a.$ Pick $i\geq 2$ such that $|u_i|\leq |u_{i+1}|.$ Then  as each $u_i$ is a prefix of $x,$ it follows that $u_i$ is a prefix of $u_{i+1}$and hence $au_iu_i$ is a factor of $x.$  Writing
$u_i=va,$ (with $v$ empty or in $\{0,1\}^+),$ we have $au_iu_i=avava $ which is an overlap, contradicting that $x$ is overlap-free.  
This proves that there exists a separating $3$-colouring for the Thue-Morse word. Recently S. Avgustinovich and O. Parshina~\cite{AP} proved that it is possible to colour $\{0,1\}^+$ using only $2$ colours in such a way that no factorisation of the Thue-Morse word is monochromatic. 

Let us remark in the example above that since the Thue-Morse word $x$ is not periodic, it follows that each proper suffix $x'$ of $x$ begins in some factor $u$ which is not a prefix of $x.$ This means that  $x'$  may be written as an infinite concatenation $x'=u_1u_2u_3\cdots$ where  $\varphi(u_i)=2$ for all $i\geq 1.$ So while $x$ itself does not admit a $\varphi$-monochromatic factorisation, where $\varphi$ is the $3$-colouring of $\{0,1\}^+$ defined above, it turns out that every proper suffix of $x$ does admit a $\varphi$-monochromatic factorisation.  Moreover, this monochromatic factorisation of $x'$ has an even stronger monochromatic property: The set $\{u_n: n\geq 1\}^+$ is $\varphi$-monochromatic  (each element has $\varphi$ colour equal to $2).$ This is because each element of $\{u_n: n\geq 1\}$ is a non-prefix of $x$ and hence the same is true of any concatenation formed by elements from this set. It turns out that a weaker version of  this phenomenon is true in greater generality:  Given any $x\in A^\nats$ and any finite colouring  $ \varphi:  \A ^+   \rightarrow C ,$ one can always find a suffix $x'$ of $x$ which admits a factorisation $x'=u_1u_2u_3\cdots$ where $\varphi(u_i\cdots u_j)=\varphi(u_1)$ for all $1\leq i\leq j.$  This fact may be obtained via a straightforward application of the infinite Ramsey theorem \cite{Ram}  (see \cite{BTC, DZ1} or  see \cite{schutz} for a proof by M. P. Sch\"utzenberger which does not use Ramsey's theorem). 

Thus Question~\ref{conj} belongs to the class of Ramsey type problems in which one tries to show that some abstract form of Ramsey's theorem does not hold in certain settings.  For instance, the infinite version of Ramsey's theorem \cite{Ram} (for colouring of pairs) states that whenever the set $\Sigma_2(\nats)$ of all $2$-element subsets of $\nats$ is finitely coloured, there exists an infinite set $X\subseteq \nats$ with $\Sigma_2(X)$ monochromatic. Hence the same applies when $\nats$ is replaced by $\reals.$  On the other hand, W. Sierpi\'nski  \cite{Serp} showed that there exists a finite colouring of  $\Sigma_2(\reals)$  such that there does not exist an uncountable set $X$ with   $\Sigma_2(X)$  monochromatic. In other words, Ramsey's theorem does not extend to the uncountable setting in $\reals.$ Similarly, by a straightforward application of Ramsey's theorem, one deduces that given any finite colouring of $\nats,$ there exists an infinite $X\subseteq \nats$ all of whose pairwise sums $\{n+m: n,m\in X, n\neq m\}$ is monochromatic. Again it follows the same is true with $\nats$ replaced by $\reals.$ On the other hand, N. Hindman, I. Leader and D. Strauss \cite{HLS} recently exhibited (using the Continuum Hypothesis CH) the existence of a finite colouring of $\reals$ such that there does not exist an uncountable set with all its pairwise sums monochromatic.  In other words, this additive formulation of Ramsey's Theorem also fails in the  uncountable setting in $\reals.$ A related question in these sorts of problems concerns the least number of colours necessary to avoid the presence of monochromatic subsets of a certain kind. For instance N. Hindman \cite{H} showed that there exists a $3$-colouring of $\nats$ such that there does not exist an infinite subset $X$ with $X+X$ monochromatic, and it is an open question of Owings \cite{Ow} whether the same result may be obtained with only $2$ colours. Again, using CH and a few more colours ($288$ to be precise), it is possible to extend Hindman's result to the reals (see Theorem 2.8 in \cite{HLS}). But again it is not known whether  the same result can be obtained with only $2$ colours~\cite{Leader}.
Thus a stronger version of Question~\ref{conj} would read:  Does every non-periodic word admit a separating $2$-colouring?

Various partial results in support of an affirmative answer to Question~\ref{conj} were obtained in \cite{BPS, DPZ, DZ1, DZ2, ST}. For instance, in \cite{DPZ}, it is shown that Question~\ref{conj} admits an affirmative answer for all non-uniformly recurrent words and various classes of uniformly recurrent words including Sturmian words. In \cite{ST},  V. Salo and I. T\"orm\"a  prove that for every aperiodic linearly recurrent word $x\in \A^\nats$ there exists a finite colouring of $\A^+$ relative to which $x$ does not admit a monochromatic factorisation into factors of increasing lengths. And recently A. Bernardino, R. Pacheco and M. Silva \cite{BPS} prove  that Question~\ref{conj}  admits an affirmative answer for all fixed points of primitive strongly recognizable substitutions. In addition to the fact that these partial  results concern only restricted classes of non-periodic words (e.g., Sturmian words or  certain fixed points of primitive substitutions),   in most cases the number of colours required to colour $\A^+$ in order to avoid a monochromatic factorisation of $x$ is found to be quite high. For instance in \cite{ST}, the authors prove that if  $x\in \A^\nats$ is an aperiodic linearly recurrent word, then there exists a constant $K\geq 2$ and a  colouring $\varphi: \A^+\rightarrow C$  with  $\Card(C)=2+\sum_{i=0}^{K^5-1}2K^i(K+1)^{2i},$  such that no factorisation of $x=u_1u_2u_3\cdots ,$ verifying the additional constraint that  $|u_i|\leq |u_{i+1}|$ for each $i\geq 1,$ is $\varphi$-monochromatic. 
The constant $K$ above is chosen such that for every factor $u$ of $x,$ every first return $w$ to $u$ satisfies $|w|\leq K|u|$ (see for instance \cite{DHS}). A similar large bound depending on the recognizability index of a substitution is obtained in \cite{BPS} in the context of fixed points of strongly recognizable substitutions. In contrast, it is shown in \cite{DPZ} that every Sturmian word admits a separating $3$-colouring. 

In this paper we give a complete and optimal affirmative answer to Question~\ref{conj} by showing that for every non-periodic word $x=x_1x_2x_3\cdots \in \A^\nats,$ there exists a $2$-colouring $\varphi: \A^+\rightarrow \{0,1\}$ relative to which no factorisation of $x$ is $\varphi$-monochromatic. Moreover, this is a  characterization of periodicity of infinite words:

\begin{thm}\label{M} Let $x=x_1x_2x_3\cdots \in \A^\nats$ be an infinite word. Then $x$ is periodic if and only if for every $2$-colouring
$\varphi: \A^+\rightarrow \{0,1\}$ there exists a $\varphi$-monochromatic factorisation of $x,$ i.e., a factorisation $x=u_1u_2u_3\cdots$ such that $\varphi(u_i)=\varphi(u_j)$ for all $i,j\geq 1.$ 
\end{thm}

Theorem~\ref{M} has several nice immediate consequences: For instance, fix a symbol $a\in \A$ and suppose an infinite word $x\in \A^\nats$ admits a  factorisation $x=u_1u_2u_3\cdots$ where each $u_i$ is a prefix of $x$ rich in the symbol $a\in \A,$ meaning
that for each $i\geq 1$ and for each factor $v$ of $x$ with $|v|=|u_i|$ we have $f_a(u_i)\geq f_a(v)$ where $f_a(u_i)$ denotes the frequency of $a$ in $u_i.$ Then $\{n\in \nats: x_n=a\}$ is a finite union of arithmetic progressions. As another application of Theorem~\ref{M} we show that if  $u_1,u_2,u_3,\ldots, u_{2k+1}\in A^+$ (with $k$ a positive integer), then any  $x\in \A^\nats$ belonging  to  $\{u_1,u_2\}^\nats \cap \{u_2,u_3\}^\nats\cap\cdots \cap \{u_{2k},u_{2k+1}\}^\nats\cap \{u_{2k+1},u_1\}^\nats$ is periodic.

Theorem~\ref{M} may be reformulated in the language of ultrafilters as follows: Let $\beta \A^+$ denote the Stone-\v Cech compactification of the discrete semigroup $\A^+$ which we regard as  the set of all ultrafilters on $\A^+$, identifying the points of
$\A^+$ with the principal ultrafilters. As is well known, the operation of concatenation on $\A^+$ extends uniquely to $\beta \A^+$ making $\beta \A^+$ a compact right topological semigroup with $\A^+$ contained in its topological center (see for instance \cite{HS}). In particular by the Ellis-Numakura lemma, $\beta\A^+$ contains an idempotent element  i.e., an element $p$ verifying $p\cdot p=p.$ We show that an infinite word $x=x_1x_2x_3\cdots \in \A^\nats$ is periodic if and only if there exists  $p\in \beta \A^+$ such that for each $A\in p$ there exists a factorisation $x=u_1u_2u_3\cdots $ with each $u_i \in A.$ Moreover $p$ can be taken to be an idempotent element of $\beta \A^+.$

\section{Proof of Theorem~\ref{M} \& Applications}
Before embarking on the proof of Theorem~\ref{M}, we introduce some notation which will be relevant in what follows. 
Let $\A$ be a non-empty set (the {\it alphabet}). We do not assume that the cardinality of $\A$ is finite. Let $\A^*$  denote the set of all finite  words $u=u_1u_2\cdots u_{n}$ with
$u_i\in \A$.   We call $n$ the {\it length} of $u$ and denote it $|u|$. The empty word is denoted $\varepsilon$ and by convention
$|\varepsilon|=0$. We put $\A^+=\A^*\setminus \{\varepsilon\}$. For $u\in \A^+$ and $a\in \A,$ we let $|u|_a$ denote the number of occurrences of $a$ in $u.$ 
Let $\A^\nats$ denote the set of all right sided infinite words $x=x_1x_2x_3\cdots$ with values in  $\A.$ 
More generally, for $x\in \A^\nats$ and $A\subseteq \A^+,$ we write $x\in A^\nats$ if $x=u_1u_2u_3\cdots $ with $u_i\in A$ for all $i\geq 1,$ that is in case $x$ factors over the set $A.$  We say $x$ is {\it periodic} if $x\in \{u\}^\nats$ for some $u\in \A^+.$  

\begin{proof}[Proof of Theorem~\ref{M}]  First assume $x=x_0x_1x_2\cdots \in \A^\nats$ is periodic, i.e., $x\in \{u\}^\nats$ for some $u\in \A^+.$ Then the factorisation $x=u_1u_2u_3\cdots$ with each $u_i=u$ is $\varphi$-monochromatic for any choice of $\varphi: \A^+\rightarrow \{0,1\}.$   Next assume $x$ is not periodic and we will define a $2$-colouring $\varphi: \A^+\rightarrow \{0,1\}$ with the property that no factorisation of $x$ is $\varphi$-monochromatic.  Pick any total order on the set $\A$ and let $\prec$ denote the induced lexicographic order on $\A^+$ and $\A^\nats.$ For $u,v\in \A^+$ with $|u|=|v|,$ we write  $u\preccurlyeq v$ if either $u\prec v$ or $u=v.$  For each $n\geq 1,$ let $P_x(n)$ denote the prefix of $x$ of length $n,$ and for each  $y \in \A^\nats,$ let $x\wedge y$ denote the longest common prefix of $x$ and $y.$

\noindent For  $u \in A^+, $ set 

\begin{equation}\label{E} \varphi(u) = \begin{cases}
0 & \text{if $u \prec P_x(|u|)$ or if $u=P_x(|u|)$ and $x\prec y$ where $x=uy;$} \\ 1 &  \text{if $P_x(|u|) \prec u$ or if $u=P_x(|u|)$ and $y\prec x$ where $x=uy.$}
\end{cases}\end{equation} 

\noindent As $x$ is not periodic, for every proper suffix $y$ of $x$ we have either $x\prec y$ or $y\prec x,$ whence $\varphi: \A^+\rightarrow \{0,1\}$ is well defined. 

We claim that no factorisation of $x$ is $\varphi$-monochromatic. To see this, fix a factorisation $x=u_1u_2u_3\cdots .$
For each $k\geq 0$ put $y_k=u_{k+1}u_ {k+2}u_{k+3}\cdots \in \A^\nats$ and let $w_k=x \wedge y_k.$ We note that
$y_0=x=w_0.$ For each $k\geq 0$ let 
\[S_k=|u_1| + |u_2|+\cdots + |u_k| + |w_k|\]
so that $S_0=|w_0|=+\infty.$ Clearly $S_k\geq k$ for each $k\geq 0$ and hence $\lim S_k=+\infty.$ 

\begin{lemma}\label{L1} Assume that $\varphi (u_i)=0$ for all $i\geq 1.$ Then $\varphi(u_1u_2\cdots u_k)=0$ for all $k\geq 1.$ 
\end{lemma}

\begin{proof} We proceed by induction on $k.$ The result is clear for $k=1$ since  $\varphi(u_1)=0. $ For the inductive step, fix $k\geq 1$ and suppose $\varphi(u_1\cdots u_k)=0.$   As $u_1\cdots u_k$ is a prefix of $x$ and $\varphi(u_1\cdots u_k)=0,$ we can write $x=w_kay'$ and $x=u_1u_2\cdots u_kw_kby''$ for some $y', y'' \in \A^\nats$ and $a,b\in \A$ with $a\prec b.$
It follows that $|u_{k+1}|\leq |w_k|$ for otherwise $w_kb$ is a prefix of $u_{k+1}$ and we would have $P_x(|u_{k+1}|)\prec u_{k+1}$ whence $\varphi(u_{k+1})=1,$ a contradiction. Thus we can write $w_k=u_{k+1}v$ for some $v\in \A^*$ so that  $x=u_{k+1}vay'$ and $x=u_1u_2\cdots u_ku_{k+1}vby''.$ Since $u_{k+1}$ is a prefix of $x$ and $\varphi(u_{k+1})=0,$ it follows that $P_x(|vb|)=P_x(|va|)\preccurlyeq va \prec vb$ which combined with the fact that $u_1\cdots u_{k+1}$ is a prefix of $x$ implies that $\varphi(u_1\cdots u_{k+1})=0$ as required. \end{proof}

\begin{lemma}\label{L2} Assume that $\varphi (u_i)=0$ for all $i\geq 1.$ Then $S_k\leq S_{k-1}$ for all $k\geq 1.$ 
\end{lemma}

\begin{proof} We first note that $S_1\leq S_0$ since $S_0=+\infty.$ Now fix $k\geq 1;$ we will show that $S_{k+1}\leq S_k.$ By the previous lemma we have that $\varphi(u_1u_2\cdots u_k)=0.$ As in the proof of the previous lemma, we can write $x=u_{k+1}vay'$ and $x=u_1u_2\cdots u_ku_{k+1}vby'' $ for some $y', y'' \in \A^\nats,$  $v\in \A^*$ and $a,b\in \A$ with $a\prec b.$ We claim that $|w_{k+1}|\leq |v|.$  In fact, if  $|w_{k+1}|>|v|,$ then $vb$ would be a prefix of $x,$ and hence $va\prec vb =P_x(|va|)$ which in turn implies
that $\varphi(u_{k+1})=1,$ a contradiction. Thus  $|w_{k+1}|\leq |v|=|w_k|-|u_{k+1}|$  or equivalently $|u_{k+1}| + |w_{k+1}| \leq |w_k|$ from which it follows that $S_{k+1}\leq S_k$ as required.\end{proof}

Returning to the proof of Theorem~\ref{M},  we must show that the factorisation $x=u_1u_2u_3\cdots $ is not $\varphi$-monochromatic. 
Suppose to the contrary that $\varphi(u_i)=\varphi(u_j)$ for all $i,j\geq 0.$ If $\varphi(u_i)=0$ for all $i\geq 1,$ then by the previous lemma we have that $S_k\leq S_{k-1}$ for all $k\geq 1$ contradicting that $\lim S_k=+\infty.$ If on the other hand $\varphi(u_i)=1$ for all $i\geq 1,$ then by replacing the original order on $A$ by the reverse order we would have $\varphi(u_i)=0$ for all $i\geq 1,$ and hence as above the previous lemma yields the desired contradiction. This concludes our proof of Theorem~\ref{M}. \end{proof}

We end  this section with some applications of Theorem~\ref{M} or its proof.  An infinite word $x\in \A^\nats$ is called {\it Lyndon} if there exists an order on $\A$ relative to which $x$ is strictly smaller  than each of its proper suffixes. Our first application is the following  result originally proved in \cite{DPZ}:

\begin{corollary}\label{C1} A Lyndon word $x\in \A^\nats$ does not admit a prefixal factorisation, i.e., a factorisation of the form $x=u_1u_2u_3\cdots$ where each $u_i$ is a prefix of $x.$
 \end{corollary}
 
 \begin{proof} Suppose to the contrary that $x=u_1u_2u_3\cdots$ is a prefixal factorisation of $x.$  Let $\varphi: \A^+ \rightarrow \{0,1\}$ be the separating $2$-colouring for $x$ defined in (\ref{E}). Note that $x$  being Lyndon is not-periodic. Since $x$ is Lyndon we have that $\varphi(u_i)=0$ for all $i\geq 1,$ contradicting that $\varphi$ is a separating $2$-colouring for $x.$   \end{proof}
 
Let $x\in \A^\nats$ and $a\in \A.$  A factor $u$ of $x$ is said to be {\it rich} in $a$ if $|u|_a\geq |v|_a$ for all factors $v$ of $x$ with $|v|=|u|.$ Theorem~6.7 in \cite{DPZ} states that a Sturmian word $x\in \{0,1\}^\nats$ does not admit a  factorisation of the form $x=u_1u_2u_3\cdots$ where each $u_i$ is a prefix of $x$ rich in the same letter $a\in \{0,1\}.$  The following generalises this result to all binary non-periodic words:

 \begin{corollary}\label{C2}Let $x\in \{0,1\}^\nats$ and $a\in \{0,1\}.$ Suppose $x$ admits a prefixal factorisation $x=u_1u_2u_3\cdots$ with each $u_i$ rich in $a.$ Then $x$ is periodic. 
 \end{corollary}
 
 \begin{proof}Suppose to the contrary that $x$ is not periodic. Let $\varphi:\{0,1\}^+\rightarrow \{0,1\}$ be the separating $2$-colouring for $x$ defined in (\ref{E}) relative to the order on $\{0,1\}$ where $a$ is taken to be the least element. For each $i\geq 1,$ writing
 $x=u_iy_i$ with $y_i\in \{0,1\}^\nats,$ we claim that $x\prec y_i.$ Otherwise if $y_i\prec x,$ then we can write
 $x=zbx'=u_izay'$ for some $z\in \{0,1\}^*,$ $x',y'\in \{0,1\}^\nats$ and where $\{a,b\}=\{0,1\}.$ But then the factor $u_i'$ of length $|u_i|$ immediately preceding the suffix $y'$ of $x$ would contain one more occurrence of the symbol $a$ than $u_i,$ contradicting that
 $u_i$ was rich in $a.$  Having established that $x\prec y_i,$ it follows that $\varphi(u_i)=0$ for all $i \geq 1$ contradicting that $\varphi$ is a separating $2$-colouring for $x.$ \end{proof}

\begin{corollary}\label{C3}Let $\A$ be an arbitrary non-empty set, $x\in \A^\nats$ and $a\in \A.$ Suppose $x$ admits a prefixal factorisation $x=u_1u_2u_3\cdots$ with each $u_i$ rich in $a.$ Then $\{n\in \nats: x_n=a\}$ is a finite union of (infinite) arithmetic progressions. 
 \end{corollary}
 
 \begin{proof} Consider the morphism $\phi: \A\rightarrow \{0,1\}$ given by $\phi(a)=1$ and $\phi(b)=0$ for all $ b\in \A\setminus \{a\}.$
 Then applying $\phi$ to the prefixal factorisation $x=u_1u_2u_3\cdots$ gives a prefixal factorisation of $\phi(x)$ in which each $\phi(u_i)$ is rich in $1.$ Hence by Corollary~\ref{C2} we have that $\phi(x)$ is periodic and hence the set of occurrences of $1$ in $\phi(x),$ which is equal to the set of occurrences of $a$ in $x,$  is a finite union of arithmetic progressions.  \end{proof}

\begin{corollary}\label{C4} Let $x=x_1x_2x_3\cdots \in \A^\nats$ and $k$ be a positive integer.  Let $B\subseteq \A^+$ with $\card(B)\geq 2k-1.$ Suppose that $x\in A^\nats$ for every $k$-element subset $A$ of $B.$ Then $x$ is periodic. \end{corollary}

\begin{proof} Let us assume to the contrary that $x$ is not periodic. By Theorem~\ref{M} there exists a $2$-colouring $\varphi:\A^+\rightarrow \{0,1\}$ relative to which no factorisation of $x$ is $\varphi$-monochromatic.  Let $\Sigma_k(B)$ denote the set of all $k$-element subsets of $B.$ By assumption $x$ factors over each $A\in \Sigma_k(B).$  On the other hand, since $\card(B)\geq 2k-1,$ it follows that there exists a $\varphi$-monochromatic subset $A\in  \Sigma_k(B).$ This gives rise to a $\varphi$-monochromatic factorisation of $x,$ a contradiction.  
\end{proof}

\begin{corollary}\label{C5} Let $x=x_1x_2x_3\cdots \in \A^\nats$ and $k$ be a positive integer.  Let $u_1,u_2,u_3,\ldots, u_{2k+1}\in A^+$ and  suppose $x\in \{u_1,u_2\}^\nats \cap \{u_2,u_3\}^\nats\cap\cdots \cap \{u_{2k},u_{2k+1}\}^\nats\cap \{u_{2k+1},u_1\}^\nats.$ Then $x$ is periodic. \end{corollary}

\begin{proof} Suppose to the contrary that $x$ is not periodic. Pick any separating $2$-colouring $\varphi: \A^+\rightarrow \{0,1\}$ for $x.$ Then $\varphi(1)=\varphi(2i+1)$ for each $1\leq i\leq k.$  Thus $x\notin \{u_{2k+1},u_1\}^\nats,$  a contradiction. 
\end{proof}

\section{A reformulation of Theorem~\ref{M} in the language of ultrafilters}

In this section give an equivalent reformulation of Theorem~\ref{M} in terms of  ultrafilters and the Stone-\v Cech compactification of the discrete semigroup $\A^+.$ 
We begin by recalling some basic facts. For more information we refer the reader to \cite{HS}.  
Let  $S$ be a non-empty set and let $\mathscr{P}(S)$ denote the set of all subsets of $S.$  A set $p\subseteq \mathscr{P}(S)$ is called a {\it filter} on $S$ if 
\begin{itemize} 
\item $S\in p$ and $\emptyset \notin p$
\item If $A\in p$ and $B\in p$ then $A\cap B\in p$
\item If $A\subseteq B$ and $A\in p$ then $B\in p.$
\end{itemize}
A filter $p$ on $S$ is called an {\it ultrafilter} if for all $A\in \mathscr{P}(S)$ either $A\in p$ or $A^c\in p$ where $A^c$ denotes the complement of $A,$ i.e., $A^c=S\setminus A.$ Equivalently,  a filter $p$ is an ultrafilter if for each $A\in p$ whenever $A=A_1\cup \cdots \cup A_n$ we have that at least one $A_i\in p.$  Each $x\in S$ determines an ultrafilter $e(x)$ on $S$ defined by $e(x)=\{A\subseteq S: x\in A\}.$ An ultrafilter $p$ on $S$ is called {\it principal} if
$p=e(x)$ for some $x\in S.$ Otherwise $p$ is said to be {\it free}.    
Let $\beta S$ denote the collection of all ultrafilters $p$ on $S.$ By identifying each $x\in S$ with the principal ultrafilter $e(x),$ we regard $S\subseteq \beta S.$   
If $S$ is infinite, then a straightforward application of Zorn's lemma guarantees the existence of free ultrafilters on $S.$
   
Given $A\subseteq S$, we put $\overline A=\{p\in\beta S:A\in p\}$. Then  $\{\overline A:A\subseteq S\}$
defines a basis for a topology on $\beta S$ relative to which $\beta S$ is both  compact and Hausdorff  and the mapping
$x\mapsto e(x)$ defines an injection $S\hookrightarrow \beta S$ whose image is dense in $\beta S.$   In fact, if $S$ is given the discrete topology, then $\beta S$ is identified with the Stone-\v Cech compactification of $S:$  Any continuous mapping from $f: S\rightarrow K,$ where $K$ is a compact Hausdorff space, lifts uniquely to a continuous mapping $\beta f: \beta S \rightarrow K.$  Of special interest is the case in which $S$ is a discrete semigroup.   In this case the operation  on $S$ extends uniquely to $\beta S$ making $\beta S$ a right topological semigroup with $S$ contained in its topological center. This means that
$\rho_p:\beta S\rightarrow \beta S,$ defined by $\rho_p(q)=q\cdot p,$ is continuous for each $p\in\beta S$ and $\lambda_s:\beta S\rightarrow \beta S,$ defined by $\lambda_s(q)=s\cdot q,$ is continuous for each $s\in S.$
The operation $\cdot$ on $\beta S$ is defined as follows:  For $p,q\in \beta S$
\[p\cdot q=\{A\subseteq S: \{s\in S:s^{-1}A\in q\}\in p\},\] where $s^{-1}A=\{t\in S:st\in A\}.$
As a consequence of the  Ellis-Numakura lemma, $\beta S$ contains an idempotent element  i.e., an element $p$ verifying $p\cdot p=p$ (see for instance \cite{HS}). Subsets  $A\subseteq S$ belonging to idempotents in $\beta S$ have rich combinatorial structures: Let $\rm{Fin}(\nats)$ denote the set of all finite subsets of $\nats.$ Given an infinite sequence $\langle s_n\rangle_{n\in \nats}$ in $S,$ 
let 
\[FP\left(\langle s_n\rangle_{n\in \nats}\right)=\{\prod_{n\in F} s_n: F\in \rm{Fin}(\nats)\}\] 
where for each $F\in \rm{Fin}(\nats),$ the product $\prod_{n\in F} s_n$ is taken in increasing order of indices.  A subset $A$ of $S$ is called an {\it IP-set} if $A$ contains $FP\left(\langle s_n\rangle_{n\in \nats}\right)$ for some infinite sequence $\langle s_n\rangle_{n\in \nats}$ in $S.$ IP-sets are characterised as belonging to idempotent elements:  $A\subseteq S$ is an IP-set if and only if $A$ belongs to some idempotent element of $\beta S$ (see for instance Theorem~5.12 in \cite{HS}).  

\noindent The following is a reformulation of Theorem~\ref{M} in the language of ultrafilters:

\begin{thm}\label{UF}  Let $x=x_1x_2x_3\cdots \in \A^\nats$ be an infinite word.  Then $x$ is periodic if and only if there exists  $p\in \beta \A^+$ such that for each $A\in p$ there exists a factorisation $x=u_1u_2u_3\cdots $ with each $u_i \in A.$ 
\end{thm}

\begin{proof} Suppose $x$ is periodic, i.e.,  $x\in \{u\}^\nats$ for some $u\in \A^+. $ Then the principal ultrafilter $e(u)=\{A\subseteq \A^+: u\in A\}$ obviously verifies the required condition of Theorem~\ref{UF}. Conversely, suppose there exists $p\in \beta \A^+$ such that for each $A\in p$ there exists a factorisation $x=u_1u_2u_3\cdots $ with each $u_i \in A.$ Let $\varphi: \A^+\rightarrow \{0,1\}$ be any $2$-colouring of $\A^+.$ We will show that $x$ admits a $\varphi$-monochromatic factorisation. The result then follows from Theorem~\ref{M}. Consider the partition $\A^+=\varphi^{-1}(0)\cup \varphi^{-1}(1).$ Since $\A^+\in p,$ it follows that $\varphi^{-1}(a)\in p$ for some $a\in \{0,1\}.$ Thus there exists a factorisation $x=u_1u_2u_3\cdots $ with each $u_i\in \varphi^{-1}(a).$ In other words, $x$ admits a $\varphi$-monochromatic factorisation.\end{proof}

\begin{remark}{\rm It can be shown that Theorem~\ref{UF} is actually an equivalent reformulation of Theorem~\ref{M}. More precisely, given an infinite word $x=x_1x_2x_3\cdots \in \A^\nats,$ the following statements are equivalent:
\begin{enumerate}
\item[a)] There exists  $p\in \beta \A^+$ such that for each $A\in p$ we can write $x=u_1u_2u_3\cdots $ with each $u_i \in A.$ 
\item[b)] For each finite colouring $\varphi:\A^+\rightarrow C$ there exists a $\varphi$-monochromatic factorisation of $x.$ 
\end{enumerate}
To see that $a)\Longrightarrow b),$ pick $p$ as in a) and let $\varphi:\A^+\rightarrow C$ be any finite colouring of $\A^+.$ Consider the partition $\A^+=\bigcup_{c\in C}\varphi^{-1}(c).$ Then there exists some $c\in C$ such that $\varphi^{-1}(c)\in p.$ This implies that we can write $x=u_1u_2u_3\cdots$ with each $u_i\in  \varphi^{-1}(c).$ In other words, $x$ admits a $\varphi$-monochromatic factorisation. Conversely to see that $b)\Longrightarrow a),$  assume b) and define $\phi: \mathscr{P}(\A^+)\rightarrow \{0,1\}$ by $\phi(A)=1$ if and only if for every finite partition $A=A_1\cup A_2\cup \cdots \cup A_n$ there exists $1\leq i\leq n$ with $A_i\in \F(x).$ Then $\phi$ verifies the following three conditions for all subsets $A,B \in \mathscr{P}(\A^+):$ \item (1) $\phi(\A^+)=1$ (this is a consequence of b). \item (2) $A\subseteq B$ and $\phi(A)=1$ then $\phi(B)=1.$ \item (3) If $\phi (A)=\phi(B)=0$ then $\phi (A\cup B)=0.$

\noindent By a standard application of Zorn's lemma, it is shown that  there exists $p\in \beta \A^+$ such that $\phi(A)=1$ for all $A\in p$ (see Theorem~3.11 in \cite{HS}).
Finally for each $A\in p,$ by considering the trivial partition $A=A,$ we deduce that $A\in \F(x)$ as required.}
\end{remark}

We next show that $p$ in Theorem~\ref{UF} may be taken to be an idempotent element of $\beta \A^+.$  For $x=x_1x_2x_3\cdots \in \A^\nats,$  we set \[\F(x)=\{A\subseteq \A^+: x\in A^\nats\}\] and \[\U(x)=\{p\in \beta \A^+: p\subseteq \F(x)\}.\]
Thus $p\in \U(x)$ if and only if $x$ factors over every $A\subseteq \A^+$ belonging to $p.$ Thus Theorem~\ref{UF} states that $\U(x)\neq \emptyset$ if and only if $x$ is periodic.

\begin{thm}\label{UF2} Let $x=x_0x_1x_2\cdots\in \A^\nats.$ Then the following are equivalent:
\begin{enumerate}
\item[i)] $x$ is periodic. 
\item[ii)] $\U(x)$ is a closed sub-semigroup of $\beta \A^+.$
\item[iii)] $\U(x)$ contains an idempotent element. 
\item[iv)] The set $\rm{Pref}(x)$ consisting of all prefixes of $x$ is an IP-set
\end{enumerate}
\end{thm} 

\begin{proof}  We first note that for each infinite word $x,$ we have that $\U(x)$ is a closed subset of $\beta \A^+.$ In fact, suppose $p\in \beta \A^+\setminus \U(x),$ then there exists $A\in p$ with $A\notin \F(x).$ Then $\overline A$ is an open neighbourhood of $p$ and any $q\in \overline A$ contains the set $A$ and hence is not in $\U(x).$ To see that $i)\Longrightarrow ii),$ suppose that  $x$ is periodic. Pick the shortest $u\in \A^+$ such that  $x\in \{u\}^\nats.$  Then the principal ultrafilter $e(u)=\{A\subseteq \A^+: u\in A\}$ clearly belongs to $\U(x).$ Thus $\U(x)\neq \emptyset.$ We note that if $x$ is periodic, then  $\U(x)$ also contains free ultrafilters. In fact, for each $i\geq 1,$ let $A_i=\{u^j: j\geq i\}.$ Then $p\in \U(x)$ if and only if $A_1\in p$ and any $p\in \beta \A^+$ containing $\{A_i: i\geq 1\}$ is a free ultrafilter belonging to $\U(x).$ It remains to show that $p\cdot q\in \U(x)$ whenever $p,q\in \U(x).$ Let $A\in p\cdot q$ with $p,q\in \U(x).$ Then $\{s\in \A^+:s^{-1}A\in q\}\in p$ and since $p\in \U(x)$ it follows that $u^n\in \{s\in \A^+:s^{-1}A\in q\}$  for some $n\in \nats.$ Thus $\{t\in \A^+: u^nt\in A\}\in q$ and since $q\in \U(x)$ it follows that $u^m \in \{t\in \A^+: u^nt\in A\}$ for some $m\in \nats.$ In other words $u^{n+m}\in A$ and hence $x$ factors over $A.$ Thus $p\cdot q\in \U(x).$
The implication $ii)\Longrightarrow iii)$ follows from the Ellis-Numakura lemma. To see $iii)\Longrightarrow iv)$ pick an idempotent element $p\in \U(x).$ We note that $\rm{Pref}(x)$ belongs to every $q\in \U(x).$ In fact, suppose that  $\rm{Pref}(x)\notin q,$ for some $q\in \U(x).$ Then $\A^+\setminus \rm{Pref}(x)\in q$ which implies that $x$ factors over  $\A^+\setminus \rm{Pref}(x).$ But this is a contradiction since in any factorisation of $x,$ the first term occurring in the factorisation belongs to $\rm{Pref}(x).$ Thus in particular $\rm{Pref}(x)\in p.$ Since $p$ is an idempotent, it follows that $\rm{Pref}(x)$ is an IP-set. Finally, to see that $iv)\Longrightarrow i)$ assume that $\rm{Pref}(x)$ is an IP-set. Then $\rm{Pref}(x)$ contains  $FP\left(\langle s_n\rangle_{n\in \nats}\right)$ for some infinite sequence $\langle s_n\rangle_{n\in \nats}$ of prefixes of $x.$ This means that for each $n\geq 2,$ both $s_1s_2\cdots s_n$ and $s_2\cdots s_n$ are prefixes of $x.$ Thus $x=s_1x$ which implies that $x$ is periodic. \end{proof}

\end{document}